\documentclass[10pt]{article}
\usepackage{amsmath,amssymb,amsthm}
\usepackage{color}
\usepackage{booktabs}
\usepackage{graphicx}
\usepackage{epstopdf}
\usepackage{algorithm,algorithmic}
\usepackage{verbatim}
\usepackage{url}
\usepackage{float}
\graphicspath{{./eps/}}
 \usepackage{array}
 \usepackage{multirow}
 \usepackage{algorithmic}
\usepackage{algorithm}
\newcolumntype{P}[1]{>{\raggedright\arraybackslash}p{#1}}

\theoremstyle{plain}
\newtheorem{theorem}{Theorem}[section]
\newtheorem{remark}{Remark}[section]

\newtheorem{proposition}{Proposition}[section]
\newtheorem{corollary}{Corollary}[section]

\numberwithin{equation}{section}

\newcommand{\R}{\mathbb{R}}
\newcommand{\N}{\mathbb{N}}

\newcommand{\E}{\mathbb{E}}

\title{Optimal Convergence of the Discrepancy Principle for polynomially and exponentially ill-posed Operators under White Noise}
\author{ Tim Jahn\thanks{Institut f\"ur Mathematik, Goethe-Universit\"at Frankfurt, Germany (\texttt{jahn@math.uni-frankfurt.de})}}

\begin{document}

\maketitle

\begin{abstract}
We consider a linear ill-posed equation in the Hilbert space setting under white noise. Known convergence results for the discrepancy principle are either restricted to Hilbert-Schmidt operators (and they require a self-similarity condition for the unknown solution $\hat{x}$, additional to a classical source condition) or to polynomially ill-posed operators (excluding exponentially ill-posed problems). In this work we show optimal convergence for a modified discrepancy principle for both polynomially and exponentially ill-posed operators (without further restrictions) solely under either H\"older-type or logarithmic source conditions. In particular, the method includes only a single simple hyper parameter, which does not need to be adapted to the type of ill-posedness.
\end{abstract}

\section{Introduction}

Let $K:\mathcal{X}\to\mathcal{Y}$ be a compact operator with dense range between infinite-dimensional Hilbert spaces. We aim to solve the following equation

$$Kx = y^\delta$$

for $y^\delta = \hat{y}+\delta Z$ a noisy perturbation of the unknown true data $\hat{y}= K\hat{x}$, with $\hat{x}=K^+\hat{y}$ the minimum norm solution (thus $K^+$ denotes the Moore-Penrose inverse of $K$). Here $Z$ is white noise, i.e. it holds that

$$\E\left[(Z,y)\right]= 0\quad\mbox{and}\quad \E\left[(Z,y)(Z,y')\right] = (y,y')$$

for all $y, y'\in\mathcal{Y}$, and $\delta>0$ denotes the noise level. We use spectral cut-off to determine an approximation to the unknown solution $\hat{x} $ and thus assume, that we know $(\sigma_j,u_j,v_j)$ the singular value decomposition of $K$, i.e. $\sigma_1\ge \sigma_2\ge ... >0$ is a decreasing sequence, $(u_j)_{j\in\N}, (v_j)_{j\in\N}$ are orthonormal bases of $\mathcal{Y}$ and $\mathcal{N}(K)^\perp \subset \mathcal{X}$ respectively, and there holds $Kv_j=\sigma_j u_j$, $K^*u_j=\sigma_j v_j$. Consequently,

$$x_k^\delta: = \sum_{j=1}^k \frac{(y^\delta,u_j)}{\sigma_j} v_j$$

is our approximation to $\hat{x}$, and the task is to determine a good choice of the truncation level $k$, dependent on the measurement $y^\delta$ and the noise level $\delta$. Because the sum $\sum_{j=1}^\infty (Z,u_j)^2$ is almost surely infinite, we cannot directly apply the discrepancy principle \cite{morozov1968error} to determine $k$. We thus truncate the sum and  discretize (as an additional regularization). Specifically, for $m\in\N$ it holds that $\E\left[\sum_{j=1}^m (y^\delta-\hat{y},u_j)^2\right] = \sum_{j=1}^m \delta^2\E(Z,u_j)^2 = m \delta^2$, so we define for fixed $\tau>1$ the classical discrepancy principle to the discretized measurements

\begin{equation}\label{int:eq1}
k^\delta_{dp}(m) :=\min\left\{0\le k\le m~:~\sqrt{\sum_{j=k+1}^m (y^\delta,u_j)^2} \le \tau \sqrt{m}\delta\right\}.
\end{equation} 

In order to determine our final approximation we have to choose the discretization level $m$.The main results (Theorem \ref{th2} and \ref{th4}) state, that for the adaptive choice

\begin{equation}\label{int:eq2}
k^\delta_{dp}:=\max_{m\in\N} k_{dp}^\delta(m)
\end{equation}

we obtain the either order optimal or even asymptotically optimal rate of convergence (in probability) under the natural source conditions for the two main types of linear ill-posed problems. These are polynomially ill-posed  and exponentially ill-posed problems (also called mildly and severely ill-posed respectively). For the sake of simplicity we assume that, in the first case, the singular values fulfill $\sigma_j^2 = j^{-q}$ for some $q>0$ and in the second case $\sigma_j^2 = e^{-aj}$ for $a>0$. Note that for polynomially ill-posed problems we do not assume that $q>1/2$, so $K$ might be a non-Hilbert-Schmidt operator. In practice the singular values of the problem are not analytically given and usually not exactly of the above specific types. So for general decaying singular values we will at least prove in Theorem \ref{th5} that the method is convergent, without giving rates. An important practical problem of our approach \eqref{int:eq1} is, that one cannot perform the maximization over all $m\in\N$ in \eqref{int:eq2}. We will comment on that and the fact that $k_{dp}^\delta$ is well-defined below.

As mentioned above, convergence rates are obtained only under additional source conditions. These are certain subspaces of $\mathcal{X}$, in which the true solution is supposed to reside. For polynomially ill-posed problems these are H\"older-type conditions (see e.g. \cite{EnglHankeNeubauer:1996} and the references therein)

\begin{equation}\label{int:eq3}
\mathcal{X}_{\nu,\rho}:=\left\{ (K^*K)^\frac{\nu}{2}\xi~:~\xi\in\mathcal{X}\right\}
\end{equation}

for $\nu,\rho>0$ (so $\hat{x}\in\mathcal{X}_{\nu,\rho}$ has the following representation $\hat{x}=\sum_{j=1}^\infty \sigma_j^{\nu}(\xi,v_j) v_j$ with $\|\xi\|\le \rho$). In the scenario of exponentially ill-posed problems it is well known, that H\"older-type smoothness conditions are often too restrictive. E.g., for the severely ill-posed problem of the inverse heat equation any H\"older-type source condition for the true solution $\hat{x}$ would imply that it is infinitely often differentiable.  Here, a natural choice are so-called logarithmic source conditions \cite{hohage2000regularization}

\begin{equation}\label{int:eq4}
\mathcal{X}_{p,\rho}=\left\{ \left(-\log(K^*K)\right)^{-\frac{p}{2}}\xi~:~\xi\in\mathcal{X},~\|\xi\|\le \rho\right\}
\end{equation}

for $p,\rho>0$. 

In the literature plenty of work has been done on linear ill-posed problems under white noise, see e.g. \cite{bissantz2007convergence} and \cite{cavalier2011inverse} for an overview. Among the first adaptive methods studied were, cross validation \cite{wahba1977practical}, unbiased or penalized empirical risk minimization \cite{cavalier2002oracle,cavalier2006risk}, Lepski's balancing principle \cite{mathe2006regularization} and others. Here optimal rates are usually obtained only up to a logarithmic factor. Also, the hyperparameters for the methods have to be chosen differently for mildly and severely ill-posed problems. Our proposed modified discrepancy principle has only one free parameter $\tau$, which can be chosen freely and independent of the degree of illposedness. This might be beneficial in practice, since usually the singular values of the problem will not behave exactly as the both cases considered here, and thus might not be classified unambiguously. This is also notable in light of the fact, that also in the deterministic case the discrepancy principle has to be adapted properly for exponentially ill-posed problems, see e.g. \cite{hohage2000regularization}. Moreover, we do not loose a logarithmic factor in the convergence rates (which is mainly due to the fact that we consider a different type of convergence). More recently variants of the discrepancy principle were studied for statistical inverse problems. In \cite{blanchard2012discrepancy,lu2014discrepancy}, a modified discrepancy principle was introduced. It is based on symmetrization and thus restricted to Hilbert-Schmidt operators. Also, the true solution $\hat{x}$ has to fulfill a self-similarity condition. Relatively new approaches in \cite{blanchard2018optimal,blanchard2018early} also use discretization to apply the discrepancy principle, as we do here. There the main goal was to minimize computational costs. Optimal rates are achieved only for polynomially ill-posed problems and the true solution must not be too smooth. In particular, the method will not work directly for exponentially ill-posed problems, as explained e.g. in Remark 3.9 of \cite{blanchard2018early}. However, the computational costs there are substantially smaller than in our case. It would be interesting whether one could combine the benefits of both methods.
Finally, we want to mention that in the above works usually bounds in $L^2$ (a.k.a in integrated mean squared error) are provided, under the assumption that the white noise is Gaussian (e.g., that $\E(Z,y)$ is Gaussian for all $y\in\mathcal{Y}$). This is often the reason for the logarithmic correction term in the rates mentioned above. We assume solely a finite second moment, but provide only rates which hold with high probability. It would be interesting whether the approach could be adapted such that in provides $L^2$ rates under Gaussian noise. Also note, that we assume that the noise level $\delta$ is ad hoc known. In case we have access to multiple measurements we may drop this assumption and use the average of those measurements as our data, and estimate the noise level in a natural way, see e.g. \cite{harrach2020beyond,harrach2020regularising,jahn2021increasing}.

\begin{remark}
First of all it is not directly clear, that $k_{dp}^\delta$ is well-defined. However, since

$$\sqrt{\sum_{j=1}^{m}(y^\delta,u_k)^2}\le \sqrt{\sum_{j=1}^{m}(\hat{y},u_k)^2}+\sqrt{\sum_{j=1}^m(\hat{y}-y^\delta,u_j)^2} \approx \|\hat{y}\| + \sqrt{m} \delta \le \tau \sqrt{m} \delta$$

for $m$ large, we see that $k_{dp}^\delta(m) \to 1$ (a.s.) as $m\to\infty$ (compare to the proof of Proposition \ref{prop0}). This assures that $k_{dp}^\delta<\infty$ a.s. and moreover, that there exists (random) $m(\delta,\|\hat{y}\|)$ with

$$\max_{m\in\N} k_{dp}^\delta(m) = \max_{m\le m(\delta,\|\hat{y}\|)} k_{dp}^\delta(m).$$ 

It would be desirable to have a rough idea of how large $m(\delta,\rho)$ will be, unfortunately we cannot give a satisfying solution for that. One natural idea to obtain an upper bound would be to balance the measurement error $\sqrt{\sum_{j=1}^m(y^\delta-\hat{y},u_j)^2}$ and the discretization error $\sqrt{\sum_{j=m+1}^\infty(\hat{y},u_j)^2}$. This would mean to determine $m$ such that roughly

$$\sqrt{m} \delta \approx \sigma_m \|\hat{x}\|,$$

since $(\hat{y},u_j) = \sigma_j(\hat{x},v_j)$. However, this can only be achieved if at least an upper bound for $\|\hat{x}\|$ is available.
 Another (heuristic) possibility is based on the observation, that for $m$ small we have that $k_{dp}^\delta(m) \approx m$, whereas $k_{dp}^\delta(m) \to 1$ as $m\to\infty$. Thus we could increase $m$ gradually, until $k_{dp}^\delta(m)/m$ is small.
\end{remark}

\section{Main Results}
We formulate the first Theorem, which states that our modified discrepancy principle resembles a convergent regularization method for arbitrary compact $K$ with dense image.

\begin{theorem}\label{th5}
Assume that $K$ is compact with dense range and let $\hat{y}\in\mathcal{R}(K)$. Let $\tau>1$ and let $k_{dp}^\delta$ be the truncation level determined by the discrepancy principle as in \eqref{int:eq1} and \eqref{int:eq2}. Then, for all $\varepsilon>0$ there holds

$$\mathbb{P}\left(\|x_{k^\delta_{dp}}^\delta - \hat{x}\|\le \varepsilon\right)\to 1$$

as $\delta\to0$.

\end{theorem}

From now we restrict to polynomially and exponentially ill-posed operators.
We first calculate the optimal a priori rate, to which we afterwards compare the rate of the discrepancy principle.

\begin{theorem}\label{th1}
Assume that the problem is polynomially ill-posed, i.e. $\sigma_j^2=j^{-q}$ for some $q>0$. Then there holds

$$\inf_{k\in\N} \sup_{\substack{\hat{x} \in\mathcal{X}_{\nu,\rho}}}\sqrt{\E\|x_{k}^\delta - \hat{x}\|^2} \asymp \rho^\frac{q+1}{(\nu+1)q+1} \delta^\frac{\nu}{\nu+1+\frac{1}{q}},$$

with $\mathcal{X}_{\nu,\rho}$ given in \eqref{int:eq3}.
\end{theorem}

\begin{proof}[Proof of Theorem \ref{th1}]
The proof is standard. We split the total error in a customary way into a data propagation error and an approximation error (also called variance and bias here, cf \eqref{sec3:eq1})

\begin{align*}
\sup_{\substack{\hat{x} \in\mathcal{X}_{\nu,\rho}}}\E\|x_k^\delta - \hat{x}\|^2 &= \sup_{\hat{x} \in\mathcal{X}_{\nu,\rho}}\sum_{j=1}^k\frac{\E(y^\delta-\hat{y},u_j)^2}{\sigma_j^2} + \sup_{\hat{x} \in\mathcal{X}_{\nu,\rho}} \sum_{j=k+1}^\infty (\hat{x},v_j)^2\\
&= \delta^2 \sum_{j=1}^k j^q\E\left(Z,u_j\right)^2 + \sup_{\substack{\xi \in\mathcal{X}\\ \|\xi\|\le \rho}} \sum_{j=k+1}^\infty \sigma^{2\nu}(\xi,v_j)^2\\
&= \delta^2\sum_{j=1}^k j^{q} + \sigma_{k+1}^{q \nu} \rho^2 \asymp \delta^2 k^{q+1} + k^{-q\nu} \rho^2.
\end{align*}

The right hand side is minimized (up to a constant factor) by a choice  fulfilling

\begin{equation}\label{sec2:eq1}
k \asymp \left(\frac{\rho}{\delta}\right)^\frac{2}{(\nu+1)q+1},
\end{equation}

 which yields the rate from Theorem \ref{th1}.
\end{proof}

The above optimal a priori choice \eqref{sec2:eq1} depends on the unknown smoothness parameter $\nu$ and $\rho$ and hence is not practical. The next Theorem assures optimal adaptivity of our modified discrepancy principle, in the sense that the optimal rate from Theorem \ref{th1} holds in probability up to a constant (order-optimal convergence).

\begin{theorem}\label{th2}

Assume that the problem is polynomially ill-posed, i.e. $\sigma_j^2 = j^{-q}$ for $q>0$. Let $\tau>1$ and let $k_{dp}^\delta$ be the truncation level determined by the discrepancy principle as in \eqref{int:eq1} and \eqref{int:eq2}. Then there holds

$$\sup_{\hat{x} \in\mathcal{X}_{\nu,\rho}}\mathbb{P}\left(\|x^\delta_{k^\delta_{dp}} - \hat{x}\| \le L_{\tau,\nu,q} \rho^\frac{q+1}{(\nu+1)q+1} \delta^\frac{\nu}{\nu+1+\frac{1}{q}}\right)\to 1,
$$

as $\delta/\rho\to0$, with $L_{\tau,\nu,q}:=\left( \frac{2}{\tau-1}+1\right)^\frac{2}{(\nu+1)q}\frac{\tau+1}{2} + \left(\frac{3\tau+1}{2}\right)^\frac{\nu}{\nu+1}+1$ and $\mathcal{X}_{\nu,\rho}$ given in \eqref{int:eq3}.
\end{theorem}

The proof  is deferred to section \ref{sec:proofs}.

We now discuss the case of exponentially ill-posed problems and again calculate the optimal possible rate first.

\begin{theorem}\label{th3}
Assume that the problem is exponentially ill-posed, i.e. $\sigma_j^2 = e^{-aj}$ for $a>0$. Then there holds

$$\inf_{k\in\N} \sup_{\hat{x}\in\mathcal{X}_{p,\rho}} \sqrt{\E\|x_k^\delta-\hat{x}\|^2} = \rho  \left(-\log\left(\frac{\delta^2}{\rho^2}\right)\right)^{-\frac{p}{2}}\left(1+o(1)\right)$$

as $\frac{\delta}{\rho}\to 0$, with $\mathcal{X}_{p,\rho}$ given in \eqref{int:eq4}.

\end{theorem}

We assure adaptivity of the above modified discrepancy principle in this case. Here we even obtain asymptotically optimal convergence, i.e. the optimal rate holds (asymptotically) up to a multiplicative constant of $1$. This is important, because of the very slow convergence under logarithmic source conditions.

\begin{theorem}\label{th4}
Assume that the problem is exponentially ill-posed, i.e. $\sigma_j^2=e^{-aj}$. Let $\tau>1$ and $k_{dp}^\delta$ be the truncation level determined by the discrepancy principle as in \eqref{int:eq1} and \eqref{int:eq2}. Then there holds

$$\sup_{\hat{x}\in\mathcal{X}_{p,\rho}}\mathbb{P}\left(\|x_{k_{dp}^\delta}^\delta - \hat{x}\| \le \rho\left(-\log\left(\frac{\delta^2}{\rho^2}\right)\right)^{-\frac{p}{2}}(1+o(1))\right)\to 1,$$

as $\delta/\rho \to 0$, with $\mathcal{X}_{p,\rho}$ given in \eqref{int:eq4}.

\end{theorem}

The proofs of Theorem \ref{th3} and \ref{th4} are presented in Section \ref{sec:proofs}.

\begin{remark}
The conditions for the singular values can be weakened to $c_qj^{-q}\le \sigma_j^2\le C_qj^{-q}$ for all $j$ large enough (and $c_q,C_q,q>0$) in the case of polynomially ill-posed problems, and to $c_a e^{-aj}\le  \sigma_j^2\le C_a e^{-aj}$ for all $j$ large enough (and $c_a,C_a,a>0$) in the case of exponentially ill-posed problems. 
\end{remark}

\section{Proofs}\label{sec:proofs}

For the proofs of Theorem \ref{th5}, \ref{th2} and \ref{th4} the following proposition is central, which states that the measurement error is  highly concentrated simultaneously for all $m$ large enough. This will allow to control the measurement error in the following.

\begin{proposition}\label{prop0}
For $m_{opt}=m_{opt}(\delta/\rho)$ with $m_{opt}(\delta/\rho)\to \infty$ as $\delta/\rho \to 0$ there holds

$$\mathbb{P}\left( \sqrt{\sum_{j=1}^m(y^\delta-\hat{y},u_j)^2} \le \frac{\tau+1}{2} \sqrt{m}\delta,~\forall m\ge m_{opt}\right)\to 1$$

as $\delta/\rho\to0$. 
\end{proposition}

\begin{proof}[Proof of Proposition \ref{prop0}]
We have that

\begin{align*}
&\mathbb{P}\left( \sqrt{\sum_{j=1}^m(y^\delta-\hat{y},u_j)^2} \le \frac{\tau+1}{2} \sqrt{m}\delta,~\forall m\ge m_{opt}\right) 
=\mathbb{P}\left( \sum_{j=1}^m(y^\delta-\hat{y},u_j)^2 \le \frac{(\tau+1)^2}{4} m\delta^2,~\forall m\ge m_{opt}\right)\\
= &\mathbb{P}\left( \frac{1}{m}\sum_{j=1}^m\left((y^\delta-\hat{y},u_j)^2 - \delta^2\right) \le \left(\frac{(\tau+1)^2}{4}-1\right) \delta^2,~\forall m\ge m_{opt}\right)\\
\ge &\mathbb{P}\left( \sup_{m\ge m_{opt}} \left|\frac{1}{m} \sum_{j=1}^m\left((y^\delta-\hat{y},u_j)^2-\delta^2\right)\right| \le \frac{\tau^2-1}{4}\delta^2\right)
 = \mathbb{P}\left( \sup_{m\ge m_{opt}} \left|\frac{1}{m} \sum_{j=1}^m\left((Z,u_j)^2-1\right)\right| \le \frac{\tau^2-1}{4}\right)\\
=: &\mathbb{P}\left( \sup_{m\ge m_{opt}} \left|\frac{1}{m} \sum_{j=1}^m X_j\right| \le \frac{\tau^2-1}{4}\right),
\end{align*}

with $X_j:=(Z,u_j)^2-1$. It is $(X_j)_{j\in\N}$ an i.i.d sequence with $\E[X_j]=0$ and $\E|X_j|\le 2$. Since the sample mean $\left(S_m\right)_{m\in\N} = \left(\frac{1}{m}\sum_{j=1}^m X_j\right)_{m\in\N}$ is a reverse martingale  (16.1 in \cite{gut2013probability}), we can apply the Kolmogorov-Doob-inequality (Theorem 16.2 in \cite{gut2013probability}) and obtain

\begin{align*}
\mathbb{P}\left( \sup_{m\ge m_{opt}} \left|\frac{1}{m} \sum_{j=1}^m X_j\right| > \frac{\tau^2-1}{4}\right) &\le \frac{4}{\tau^2-1}\E\left[\left| \frac{1}{m_{opt}}\sum_{j=1}^{m_{opt}}X_j\right|\right]\to 0
\end{align*}

as $\delta\to 0$, where we have used in the last step, that $\lim_{\delta\to\infty} m_{opt}(\delta/\rho)=\infty$ and that $\E [X_j]=0$ and that the sample mean converges in $L^1$ to its expectation (Theorem 16.4 in \cite{gut2013probability}). Putting all together concludes the proof.

\end{proof}

In the proofs of all theorems we will split the total error into a data propagation error and an approximation error

\begin{align}\label{sec3:eq1}
\|x_{k}^\delta - \hat{x}\| &= \left\| \sum_{j=1}^k \frac{(y^\delta,u_j)}{\sigma_j} v_j - \sum_{j=1}^\infty(\hat{x},v_j)v_j\right\| =  \left\|\sum_{j=1}^k\left(\frac{(y^\delta,u_j)}{\sigma_j} - (\hat{x},v_j)\right) v_j - \sum_{j=k+1}^\infty(\hat{x},v_j) v_j\right\|\\\notag
  &= \left\|\sum_{j=1}^k\frac{(y^\delta-\hat{y},u_j)}{\sigma_j} v_j - \sum_{j=k+1}^\infty (\hat{x},v_j)v_j\right\|   \le \sqrt{\sum_{j=1}^{k} \frac{(y^\delta-\hat{y},u_j)^2}{\sigma_j^2}} + \sqrt{\sum_{j=k+1}^\infty (\hat{x},v_j)^2}.
\end{align}

and treat both terms individually.

\subsection{Proof of Theorem \ref{th5}}

We start with an auxiliary proposition.

\begin{proposition}\label{th5:prop1}
Let $k\in\N$ be such that $(\hat{y},u_k)\neq 0$. Then there holds

$$\mathbb{P}\left(k_{dp}^\delta \ge k\right)\to 1$$

as $\delta\to0$.
\end{proposition}

\begin{proof}
Let $m_\delta = \delta^{-1}$. It is

\begin{align}\label{th5:prop1:eq1}
\sqrt{\sum_{j=k}^{m_\delta}(y^\delta,u_j)^2} &\ge \sqrt{\sum_{j=k}^{m_\delta}(\hat{y},u_j)^2} - \sqrt{\sum_{j=k}^{m_\delta}(y^\delta-\hat{y},u_j)^2} \ge |(\hat{y},u_k)| - \sqrt{\sum_{j=k}^{m_\delta}(y^\delta-\hat{y},u_j)^2}.
\end{align}

Moreover by Markov's inequality

\begin{align}\label{th5:prop1:eq2}
\mathbb{P}\left(\sqrt{\sum_{j=k}^{m_\delta}(y^\delta-\hat{y},u_j)^2} > |(\hat{y},u_k)| - \tau \sqrt{m_\delta}\delta\right)&\le \frac{\E\left[\sum_{j=k}^{m_\delta}(y^\delta-\hat{y},u_j)^2\right]}{\left(|(\hat{y},u_k)| - \tau \sqrt{m_\delta}\delta\right)^2}\\\notag
&\le \frac{m_\delta \delta^2}{\left(|(\hat{y},u_k)| - \tau \sqrt{m_\delta}\delta\right)^2} = \frac{\delta^{-1}}{\left(|(\hat{y},u_k)| - \tau \delta^{-\frac{1}{2}}\right)^2} \to 0
\end{align}

as $\delta\to0$ (since $(\hat{y},u_k)\neq 0$). With the definition of the discrepancy principle and \eqref{th5:prop1:eq1}, \eqref{th5:prop1:eq2} we deduce

$$\mathbb{P}\left(k_{dp}^\delta(m_\delta) \ge k\right) \ge \mathbb{P}\left( \sqrt{\sum_{j=k}^{m_\delta}(y^\delta,u_j)^2} > \tau \sqrt{m_\delta}\delta\right)\ge \mathbb{P}\left(\sqrt{\sum_{j=k}^{m_\delta}(y^\delta-\hat{y},u_j)^2}>|(\hat{y},u_k)| - \tau \sqrt{m_\delta}\delta\right) \to 0$$

as $\delta\to0$, and the assertion follows with $k_{dp}^\delta\ge k_{dp}^\delta(m_\delta)$.

\end{proof}

Now we set

\begin{equation}\label{th5:eq1}
J:=\sup \left\{ j \in \N~:~ (\hat{y},u_j)\neq 0\right\}
\end{equation}

and distinguish the cases $J<\infty$ and $J=\infty$.

\subsubsection{Case 1}

We start with an easy corollary, which assures the existence of a deterministic lower bound, which holds with high probability.

\begin{corollary}\label{cor1}
Assume that $J=\infty$. Then there exists $(q_\delta)\subset \N$ with $q_\delta\nearrow\infty$ and

$$\mathbb{P}\left( k_{dp}^\delta\ge q_\delta\right)\to 1$$

as $\delta\to0$.
\end{corollary}
\begin{proof}
This follows directly from Proposition \ref{th5:prop1}.
\end{proof}

For $(q_\delta)_{\delta>0}\subset \R$ with $q_\delta\to \infty$ and $\mathbb{P}\left(k_{dp}^\delta\ge q_\delta)\right)\to1$ for $\delta\to0$  we now define 

\begin{equation}\label{th5:eq2}
\Omega_\delta:=\left\{ \sqrt{\sum_{j=1}^m(\hat{y}-y^\delta,u_j)^2} \le \frac{\tau+1}{2}\sqrt{m}\delta~\forall m\ge q_\delta~,~k_{dp}^\delta\ge q_\delta,~|(y^\delta-\hat{y},u_1)|\le \delta^{-\frac{1}{2}}\right\}.
\end{equation}

Note that $\mathbb{P}\left(\Omega_\delta\right)\to1$ because of Proposition \ref{prop0} and

$$\mathbb{P}\left(|(y^\delta-\hat{y},u_1)|^2 > \sqrt{\delta}\right)\le \frac{\E(y^\delta-\hat{y},u_1)^2}{\delta} = \delta \E(Z,u_1) \to 0$$

as $\delta\to0$. The following proposition controls the data propagation error.

\begin{proposition}\label{th5:prop2}
Assume that $J=\infty$. Then for all $m\ge q_\delta$ there holds

$$\frac{\sqrt{\sum_{j=1}^{k_{dp}^\delta}(y^\delta-\hat{y},u_j)^2}}{\sigma_{k_{dp}^\delta}} \chi_{\Omega_\delta}\le
                  \frac{\tau+1}{\tau-1}\sqrt{\sum_{j=k_{dp}^\delta}^\infty(\hat{x},v_j)^2}.$$

\end{proposition}

\begin{proof}

Since $k_{dp}^\delta\chi_{\Omega_\delta}\ge q_\delta$ there exists (random) $M$ with $M\chi_{\Omega_\delta}\ge q_\delta$ and $k_{dp}^\delta(M)\chi_{\Omega_\delta} = k_{dp}^\delta\chi_{\Omega_\delta}$. Since $M=\sum_{m=1}^\infty m \chi_{\{ M=m\}}$, it suffices to show that

$$\frac{\sqrt{\sum_{j=1}^{k_{dp}^\delta(m)}(y^\delta-\hat{y},u_j)^2}}{\sigma_{k_{dp}^\delta(m)}} \chi_{\Omega_\delta}\le \begin{cases} \frac{\sqrt{\delta}}{\sigma_{1}}&\quad k_{dp}^\delta(m)=1,\\
                  \frac{\tau+1}{\tau-1}\sqrt{\sum_{j=k_{dp}^\delta(m)}^\infty(\hat{x},v_j)^2}&\quad \mbox{else}\end{cases}$$

for all (deterministic) $m\ge q_\delta$. So let us first assume that $k_{dp}^\delta(m)=1$. Then

$$\frac{\sqrt{\sum_{j=1}^{k_{dp}^\delta(m)}(y^\delta-\hat{y},u_j)^2}}{\sigma_{k_{dp}^\delta(m)}}\chi_{\Omega_\delta} = \frac{|(y^\delta-\hat{y},u_1)|}{\sigma_1}\chi_{\Omega_\delta} \le\frac{\sqrt{\delta}}{\sigma_1}.$$

Now assume that $k_{dp}^\delta(m)\ge 2$. Then, the defining relation of the discrepancy principle is not fulfilled for $k=k_{dp}^\delta(m)-1$. Therefore

\begin{align*}
\tau \sqrt{m} \delta\chi_{\Omega_\delta} &< \sqrt{\sum_{j=k_{dp}^\delta(m)}^m(y^\delta,u_j)^2}\chi_{\Omega_\delta} \le \sqrt{\sum_{j=k_{dp}^\delta(m)}^m(\hat{y},u_j)^2} + \sqrt{\sum_{j=k_{dp}^\delta(m)}^m(y^\delta-\hat{y},u_j)^2}\chi_{\Omega_\delta}\\
                                         &\le \sigma_{k_{dp}^\delta(m)}\sqrt{\sum_{j=k_{dp}^\delta(m)}^m(\hat{x},v_j)^2} + \frac{\tau+1}{2}\sqrt{m}\delta.
\end{align*}

Rearranging the expression yields

$$ \frac{\sqrt{m}}{\sigma_{k_{dp}^\delta(m)}} \delta \le \frac{2}{\tau-1}\sqrt{\sum_{j=k_{dp}^\delta(m)}^m(\hat{x},v_j)^2}.$$

Finally,

$$\frac{\sqrt{\sum_{j=1}^{k_{dp}^\delta(m)}(y^\delta-\hat{y},u_j)^2}}{\sigma_{k_{dp}^\delta(m)}}\chi_{\Omega_\delta}\le \frac{\sqrt{\sum_{j=1}^{m}(y^\delta-\hat{y},u_j)^2}}{\sigma_{k_{dp}^\delta(m)}}\chi_{\Omega_\delta}\le \frac{\tau+1}{2}\frac{\sqrt{m}}{\sigma_{k_{dp}^\delta(m)}}\delta\chi_{\Omega_\delta}\le \frac{\tau+1}{\tau-1} \sqrt{\sum_{j=k_{dp}^\delta(m)}^m(\hat{x},v_j)^2}.$$

\end{proof}

We are now ready to prove Theorem \ref{th5} under the assumption that $J=\infty$. We split into a data propagation error and an approximation error

\begin{align*}
\|x_{k_{dp}^\delta}^\delta-\hat{x}\|\chi_{\Omega_\delta} &\le \sqrt{\sum_{j=1}^{k_{dp}^\delta}\frac{(y^\delta-\hat{y},u_j)^2}{\sigma_j^2}}\chi_{\Omega_\delta} + \sqrt{\sum_{j=k_{dp}^\delta+1}^\infty(\hat{x},v_j)^2}\chi_{\Omega_\delta}\le \frac{1}{\sigma_{k_{dp}^\delta}} \sqrt{\sum_{j=1}^{k_{dp}^\delta}(y^\delta-\hat{y},u_j)^2}\chi_{\Omega_\delta} + \sqrt{\sum_{j=q_\delta+1}(\hat{x},v_j)^2}\\
&\le \frac{\tau+1}{\tau-1} \sqrt{\sum_{j=k_{dp}^\delta}^\infty(\hat{x},v_j)^2}\chi_{\Omega_\delta}+\sqrt{\sum_{j=q_\delta+1}^\infty(\hat{x},v_j)^2}\le \frac{\tau+1}{\tau-1}\sqrt{\sum_{j=q_\delta}^\infty(\hat{x},v_j)^2}+\sqrt{\sum_{j=q_\delta+1}^\infty(\hat{x},v_j)^2},
\end{align*}

where we have multiple times used $k_{dp}^\delta\chi_{\Omega_\delta}\ge q_\delta$ and Proposition \ref{th5:prop2}. Since $q_\delta\to \infty$ it holds that 
$$\frac{\tau+1}{\tau-1}\sqrt{\sum_{j=q_\delta}^\infty(\hat{x},v_j)^2}, \sqrt{\sum_{j=q_\delta+1}^\infty (\hat{x},v_j)^2}\le \frac{\varepsilon}{2}$$
 for $\delta$ small enough, so finally

$$\mathbb{P}\left(\|x_{k_{dp}^\delta}^\delta-\hat{x}\|\le \varepsilon\right)\ge \mathbb{P}\left(\Omega_\delta\right)\to 1$$

as $\delta\to0$.

\subsubsection{Case 2} 
We now consider the case $J<\infty$. Define 

\begin{equation}\label{th5:eq3}
q_\delta:=\max\left\{q\in\N~:~ \frac{\sqrt{q}}{\sigma_q}\le \delta^{-\frac{1}{2}}\right\}.
\end{equation}

Note that $q_\delta\to \infty$ and $\frac{\sqrt{q_\delta}}{\sigma_{q_\delta}}\delta\to 0$ as $\delta\to0$. Moreover, $q_\delta$ is an upper bound for $k_{dp}^\delta$ with high probability.

\begin{proposition}\label{th5:prop3}
Assume that $J<\infty$. Then, for $q_\delta$ given in \eqref{th5:eq3} there holds

$$\mathbb{P}\left( k_{dp}^\delta \le q_\delta\right)\to1$$

as $\delta\to0$.
\end{proposition}

\begin{proof}
For $\delta$ small enough it is $q_\delta>J$, and thus $(y^\delta,u_j)=(y^\delta-\hat{y},u_j)$ for all $j\ge q_\delta$. clearly, $k_{dp}^\delta(m)\le m\le q_\delta$ for all $m\le q_\delta$. Therefore, for $\delta$ sufficiently small and $\tau'=2\tau-1$,

\begin{align*}
\mathbb{P}\left( k_{dp}^\delta \le q_\delta\right) &= \mathbb{P}\left( k_{dp}^\delta(m) \le q_\delta,~\forall m\ge q_\delta\right) = \mathbb{P}\left(\sqrt{\sum_{j=q_\delta}^m(y^\delta,u_j)^2}\le \tau \sqrt{m}\delta,~\forall m\ge q_\delta\right)\\
&= \mathbb{P}\left( \sqrt{\sum_{j=q_\delta}^m(y^\delta-\hat{y},u_j)^2} \le \tau \sqrt{m}\delta,~\forall m\ge q_\delta\right)\\
&= \mathbb{P}\left( \sqrt{\sum_{j=1}^m(y^\delta-\hat{y},u_j)^2} \le \frac{\tau'+1}{2}\sqrt{m}\delta,~\forall m\ge q_\delta\right)\to 1
\end{align*}

as $\delta\to0$, where we used $\delta$ sufficiently small in the third and Proposition \ref{prop0} (with $m_{opt}=q_\delta$ and $\tau=\tau'$) in the last step.

\end{proof}

We come to the main proof and define

$$\Omega_\delta:=\left\{ \sqrt{\sum_{j=1}^{q_\delta}(y^\delta-\hat{y},u_j)^2}\le \frac{\tau+1}{2}\sqrt{q_\delta} \delta,~ J\le k_{dp}^\delta\le q_\delta\right\}.$$

It holds that $\mathbb{P}\left(\Omega_\delta\right)\to 1$ as $\delta\to0$ because of Proposition \ref{prop0}, \ref{th5:prop1} and \ref{th5:prop3}. We split as usual

\begin{align*}
\|x_{k_{dp}^\delta}^\delta-\hat{x}\|\chi_{\Omega_\delta} &\le \sqrt{\sum_{j=1}^{k_{dp}^\delta}\frac{(y^\delta-\hat{y},u_j)^2}{\sigma_j^2}}\chi_{\Omega_\delta} + \sqrt{\sum_{j=k_{dp}^\delta+1}^\infty(\hat{x},v_j)^2}\chi_{\Omega_\delta} \le \frac{1}{\sigma_{q_\delta}}\sqrt{\sum_{j=1}^{q_\delta}(y^\delta-\hat{y},u_j)^2}\chi_{\Omega_\delta} + \sqrt{\sum_{j=J+1}^\infty(\hat{x},v_j)^2}\\
&\le \frac{\tau+1}{2}\frac{\sqrt{q_\delta}}{\sigma_\delta}\delta \to 0
\end{align*}

as $\delta\to0$ by definition of $J$ and $q_\delta$. This concludes the proof of Theorem \ref{th5}.

\subsection{Proof of Theorem \ref{th2}}
So let $\hat{x}\in\mathcal{X}_{\nu,\rho}$. We fix an (order-) optimal a priori choice from \eqref{sec2:eq1} and set

\begin{equation}\label{sec3:eq0}
m_{opt}:=m_{opt}(\delta,\rho,\nu,q) = \left\lceil \frac{\rho}{\delta}\right\rceil^\frac{2}{(\nu+1)q+1}.
\end{equation}

We define the event

\begin{equation}\label{sec3:eq3}
\Omega_{\delta/\rho}:=\left\{ \sqrt{\sum_{j=1}^m (y^\delta-\hat{y},u_j)^2} \le \frac{\tau+1}{2} \sqrt{m}\delta,~\forall m\ge m_{opt}\right\},
\end{equation}

where we have good control of the random error. It is $\lim_{\delta/\rho\to0}\mathbb{P}\left(\Omega_{\delta/\rho}\right)=1$ by Proposition \ref{prop0} above. 
We first show, that somewhat surprisingly, $k^\delta_{dp}(m)$ is bounded on $\Omega_{\delta/\rho}$ by the optimal choice from \eqref{sec2:eq1} uniformly in $m$. This implies in particular, that the same bound holds for $k_{dp}^\delta$ from which we will later conclude that the data propagation error has the optimal order.

\begin{proposition}\label{prop1}
It holds that 

$$k^\delta_{dp}(m)\chi_{\Omega_{\delta/\rho}} \le C_{\tau,\nu,q} m_{opt}(\delta,\nu,\rho)\quad\mbox{for all }m\in\N,$$

where $C_{\tau,\nu,q}:=\left(\frac{2}{\tau-1}\right)^\frac{2}{(\nu+1)q}$ and $m_{opt}(\delta,\nu,\rho)$ given in \eqref{sec3:eq0}.
\end{proposition}

\begin{proof}[Proof of Proposition \ref{prop1}]
By definition of $k^\delta_{dp}(m)$, it clearly holds that $k^\delta_{dp}(m)\le m$, so we can assume that $m\ge m_{opt}$. Moreover, we can assume that $k^\delta_{dp}(m)\ge 2$, that is the defining relation of the discrepancy principle is not fulfilled for $k^\delta_{dp}(m)-1$. Thus there holds

\begin{align*}
\tau \sqrt{m} \delta &< \sqrt{\sum_{j=k^\delta_{dp}(m)}^m(y^\delta,u_j)^2} \le \sqrt{\sum_{j=k^\delta_{dp}(m)}^m (\hat{y},u_j)^2} + \sqrt{\sum_{j=k^\delta_{dp}(m)}^m (y^\delta-\hat{y},u_j)^2}\\
&\le \sqrt{\sum_{j=k_{dp}^\delta(m)}^m(\hat{y},u_j)^2} + \sqrt{\sum_{j=1}^m(y^\delta-\hat{y},u_j)^2}.
\end{align*}

On $\Omega_{\delta/\rho}$ (cf \eqref{sec3:eq3}) we can further bound the right hand side

\begin{align*}
\tau \sqrt{m} \delta\chi_{\Omega_{\delta/\rho}} &< \sqrt{\sum_{j=k_{dp}^\delta(m)}^m(\hat{y},u_j)^2} + \sqrt{\sum_{j=1}^m(y^\delta-\hat{y},u_j)^2}\chi_{\Omega_{\delta/\rho}}\\
 &\le \sqrt{\sum_{j=k^\delta_{dp}(m)}^m \sigma_j^{2(\nu+1)} (\xi,v_j)^2} +\frac{\tau+1}{2} \sqrt{m} \delta \le k^\delta_{dp}(m)^{-\frac{(\nu+1)q}{2}} \rho + \frac{\tau+1}{2}\sqrt{m}\delta.
 \end{align*}
 
 We solve for $k_{dp}^\delta(m)$ and obtain the assertion
 
\begin{align*} 
 k^\delta_{dp}(m)\chi_{\Omega_{\delta/\rho}} &\le \left( \frac{2}{\tau-1} \frac{\rho}{\sqrt{m}\delta}\right)^\frac{2}{(\nu+1)q}  = C_{\tau,\nu,q} \left(\frac{\rho}{\sqrt{m} \delta}\right)^\frac{2}{(\nu+1)q} \le C_{\tau,\nu,q} \left(\frac{\rho}{\left(\frac{\rho}{\delta}\right)^\frac{1}{(\nu+1)q+1}\delta}\right)^\frac{2}{(\nu+1)q}\\
  &=C_{\tau,\nu,q} \left( \left(\frac{\rho}{\delta}\right)^\frac{(\nu+1)q}{(\nu+1)q+1}\right)^\frac{2}{(\nu+1)q} =C_{\tau,\nu,q} \left(\frac{\rho}{\delta}\right)^\frac{2}{(\nu+1)q+1},
\end{align*}

where we have used $m\ge m_{opt}$ in the third step.

\end{proof}

Now we treat the approximation error. Since $k^\delta_{dp} \ge k^{\delta}_{dp}(m)$ there holds that

$$\sqrt{\sum_{j=k^\delta_{dp}+1}^\infty (\hat{x},v_j)^2} \le \sqrt{\sum_{j=k^\delta_{dp}(m)+1}^\infty (\hat{x},v_j)^2}$$

for all $m\in\N$. Thus $k^\delta_{dp}$ minimizes the approximation error. In order to finish the proof of Theorem \ref{th2} it remains to show, that there is a $m\in\N$ such that $\sqrt{\sum_{j=k^\delta_{dp}(m)+1}^\infty (\hat{x},v_j)^2}$ has the optimal rate. We show, that on $\Omega_{\delta/\rho}$ this holds for our optimal a priori choice $m_{opt}(\delta,\nu,\rho)$ from \eqref{sec3:eq0}.

\begin{proposition}\label{prop2}
There holds

$$\sqrt{\sum_{j=k^\delta_{dp}(m_{opt})+1}^\infty (\hat{x},v_j)^2}\chi_{\Omega_{\delta/\rho}} \le C_{\tau,\nu} \rho^\frac{(q+1)}{(\nu+1)q + 1} \delta^\frac{\nu}{\nu+1+\frac{1}{q}},
$$

with $C_{\tau,\nu}: = \left(\frac{3\tau+1}{2}\right)^\frac{\nu}{\nu+1} +1$ and $m_{opt}$ from \eqref{sec3:eq0}.
\end{proposition}

\begin{proof}[Proof of Proposition \ref{prop2}]
We use $k_{dp}^\delta(m)\le m$ and split

\begin{align}\notag
\sqrt{\sum_{j=k^\delta_{dp}(m_{opt})+1}^\infty (\hat{x},v_j)^2} &\le \sqrt{\sum_{j=k^\delta_{dp}(m_{opt})+1}^{m_{opt}} (\hat{x},v_j)^2}+\sqrt{\sum_{j=m_{opt}+1}^\infty (\hat{x},v_j)^2}.
\end{align}

For the second term it holds that

\begin{align*}
\sqrt{\sum_{j=m_{opt}+1}^\infty (\hat{x},v_j)^2} = \sqrt{\sum_{j=m_{opt}+1}^\infty \sigma_j^{2\nu}(\xi,v_j)^2} &\le \sigma_{m_{opt}}^\nu \rho
 = m_{opt}^{-\frac{\nu q}{2}}\rho= \left(\frac{\delta}{\rho}\right)^\frac{q \nu}{(\nu+1)q+1} \rho = \delta^\frac{ \nu}{\nu+1+\frac{1}{q}} \rho^\frac{q +1}{(\nu+1)q+1}
\end{align*}

 It remains to bound the first term of the right hand side (on $\Omega_{\delta/\rho}$). For that we use a standard argumentation to bound the approximation error for the discrepancy principle and apply H\"older's inequality (for $p =\frac{\nu+1}{\nu}, q=\nu+1$)

\begin{align*}
&\sqrt{\sum_{j=k^\delta_{dp}(m_{opt})+1}^{m_{opt}} (\hat{x},v_j)^2}\chi_{\Omega_{\delta/\rho}} = \sqrt{\sum_{j=k^\delta_{dp}(m_{opt})+1}^{m_{opt}}  \sigma_j^{2\nu} (\xi,v_j)^\frac{2\nu}{\nu+1} (\xi,v_j)^\frac{2}{\nu+1}}\chi_{\Omega_{\delta/\rho}}\\
\le &\left( \sum_{j=k^\delta_{dp}(m_{opt})+1}^{m_{opt}}  \sigma_j^{2(\nu+1)} (\xi,v_j)^2\right)^\frac{\nu}{2(\nu+1)} \left( \sum_{j=k^\delta_{dp}(m_{opt})+1}^{m_{opt}}  (\xi,v_j)^2\right)^\frac{1}{2(\nu+1)}\chi_{\Omega_{\delta/\rho}}\\
 \le &\rho^\frac{1}{\nu+1} \left( \sum_{j=k^\delta_{dp}(m_{opt})}^{m_{opt}}  (\hat{y},u_j)^2\right)^\frac{\nu}{2(\nu+1)}\chi_{\Omega_{\delta/\rho}}\\
  \le  &\rho^\frac{1}{\nu+1} \left( \sqrt{\sum_{j=k^\delta_{dp}(m_{opt})}^{m_{opt}}  (y^\delta,u_j)^2} +\sqrt{ \sum_{j=k^\delta_{dp}(m_{opt})}^{m_{opt}}  (\hat{y}-y^\delta,u_j)^2}\chi_{\Omega_{\delta/\rho}}\right)^\frac{\nu}{(\nu+1)}\\
 \le &\rho^\frac{1}{\nu+1} \left( \tau \sqrt{m_{opt}} \delta + \frac{\tau+1}{2}\sqrt{m_{opt}} \delta\right)^\frac{\nu}{\nu+1}
  \le \rho^\frac{1}{\nu+1} \left(\frac{3\tau+1}{2}\right)^\frac{\nu}{\nu+1} \left(\left(\frac{\rho}{\delta}\right)^\frac{1}{(\nu+1)q+1} \delta\right)^\frac{\nu}{\nu+1}\\ 
    = &\left(\frac{3\tau+1}{2}\right)^\frac{\nu}{\nu+1} \rho^{\left(\frac{1}{\nu+1} + \frac{\nu}{((\nu+1)q+1)(\nu+1)}\right)}\left( \delta^\frac{(\nu+1)q}{(\nu+1)q+1}\right)^\frac{\nu}{\nu=1}  = \left(\frac{3\tau+1}{2}\right)^\frac{\nu}{\nu+1} \rho^\frac{q+1}{(\nu+1)q+1} \delta^\frac{\nu}{\nu+1+\frac{1}{q}}
\end{align*}

where we used the definition of $k_{dp}^\delta(m_{opt})$ for the first term in the fifth step. Together with the preceding bound this finishes the proof of Proposition \ref{prop2}.

\end{proof}

We finish the proof of Theorem \ref{th2}. For the decomposition \eqref{sec3:eq1} on $\Omega_{\delta/\rho}$ it holds that

\begin{align*}
\|x_{k^\delta_{dp}} - \hat{x}\|\chi_{\Omega_{\delta/\rho}} &\le \sqrt{\sum_{j=1}^{k^\delta_{dp}} \frac{(y^\delta-\hat{y},u_j)^2}{\sigma_j^2}}\chi_{\Omega_{\delta/\rho}} + \sqrt{\sum_{j=k^\delta_{dp}+1}^\infty (\hat{x},v_j)^2}\chi_{\Omega_{\delta/\rho}}\\
 &\le \frac{1}{\sigma_{k^\delta_{dp}}} \sqrt{\sum_{j=1}^{k^\delta_{dp}}(y^\delta-\hat{y},u_j)^2}\chi_{\Omega_{\delta/\rho}} + C_{\tau,\nu} \rho^\frac{(q+1)}{(\nu+1)q+1} \delta^\frac{\nu}{\nu+1+\frac{1}{q}}\\
  \le  &\frac{1}{\sigma_{\lceil C_{\tau,\nu,q} m_{opt}\rceil}} \sqrt{\sum_{j=1}^{m_{opt}}(y^\delta-\hat{y},u_j)^2}\chi_{\Omega_{\delta/\rho}} + C_{\tau,\nu} \rho^\frac{(q+1)}{(\nu+1)q+1} \delta^\frac{\nu}{\nu+1+\frac{1}{q}} \\
  \le & \lceil C_{\tau,\nu,q} m_{opt}\rceil^\frac{q}{2} \frac{\tau+1}{2}\sqrt{m_{opt}}\delta + C_{\tau,\nu} \rho^\frac{(q+1)}{(\nu+1)q+1} \delta^\frac{\nu}{\nu+1+\frac{1}{q}}\\
  &\le (C_{\tau,\nu,q} +1)^\frac{q}{2}\frac{\tau+1}{2} m_{opt}^\frac{q+1}{2}\delta + C_{\tau,\nu}\rho^\frac{(q+1)}{(\nu+1)q+1} \delta^\frac{\nu}{\nu+1+\frac{1}{q}}
  \le  L_{\tau,\nu,q} \rho^\frac{(q+1)}{(\nu+1)q+1} \delta^\frac{\nu}{\nu+1+\frac{1}{q}},
\end{align*}

with $L_{\tau,\nu,q}$ from Theorem \ref{th2}, where we used Proposition \ref{prop2} in the second, Proposition \ref{prop1} in the third and the definition of $\Omega_{\delta/\rho}$ \eqref{sec3:eq0} in the fourth step. Finally, Proposition \ref{prop0} implies $\mathbb{P}\left(\Omega_{\delta/\rho}\right)\to1$ as $\delta/\rho\to0$ and hence finishes the proof of Theorem \ref{th2}.

\subsection{Proof of Theorem \ref{th3}}

We start with an auxiliary proposition. For $a,b>0$ define

$$f(x):= x^b e^{a x}.$$

\begin{proposition}\label{prop1:th3}

For every $y>0$ the equation $f(x)=y$ has a unique solution $x^*$ in $(0,\infty)$. Moreover, there holds

\begin{align}
x^*&= \frac{1}{a}\log(y) - \frac{1}{a}\log\left( \left(\frac{1}{a}\log(y)\right)^b\right) + o(1)
\end{align}

as $y\to\infty$.

\end{proposition}

\begin{proof}
$f$ is continuous on $[0,\infty)$ and strictly monotonically increasing, with $f(0)=0$ and $\lim_{x\to\infty}f(x)=\infty$, this guarantees the existence of a unique solution  $x^*=x^*(y)$. Set $z=z(y):=\frac{1}{a}\log(y) - \frac{1}{a}\log\left(\left(\frac{1}{a}\log(y)\right)^b\right)$. First of all,

\begin{align*}
f\left(z\right) &= \left(\frac{1}{a}\log(y) - \frac{1}{a}\log\left(\left(\frac{1}{a}\log(y)\right)^b\right)\right)^b e^{a\left(\frac{1}{a}\log(y) - \frac{1}{a}\log\left(\left(\frac{1}{a}\log(y)\right)^b\right)\right)}\\
&=\left(\frac{1}{a}\log(y) - \frac{1}{a}\log\left(\left(\frac{1}{a}\log(y)\right)^b\right)\right)^b \frac{1}{\left(\frac{1}{a}\log(y)\right)^b} y\\
&= y \left( 1 - \frac{\log\left(\left( \frac{1}{a}\log(y)\right)^b\right)}{\log(y)}\right)^b. 
\end{align*}

 Since $y=f(x^*)$ and $f$ is monotonically increasing, the above reasoning implies $z\le x^*$. Note that $f'(x) = bx^{b-1}e^{ax} + ax^k e^{bx} \ge a f(x)$. This together with the above calculation finally yields 

\begin{align*}
0\le x^*-z &= f^{-1}(f(x^*)) - f^{-1}(f(z)) = \int_{f(z)}^{f(x^*)} (f^{-1})'(t) dt = \int_{f(z)}^{f(x^*)} \frac{1}{f'(f^{-1}(t))} dt \le \int_{f(z)}^{f(x^*)} \frac{1}{af(f^{-1}(t))} dt\\
 &\le \int_{f(z)}^{f(x^*)} \frac{1}{at}dt = \frac{1}{a} \log\left(\frac{f(x^*)}{f(z)}\right)= \frac{1}{a}\log\left(\left(1-\frac{\log\left(\left(\frac{1}{a}\log(y)\right)^b\right)}{\log(y)}\right)^{-b}\right)\to 0
\end{align*}

for $y\to\infty$, thus $x^*=z+o(1)$.

\end{proof}

We start with the proof of Theorem \ref{th3}. The usual split gives
\begin{align*}
\sup_{\hat{x}\in\mathcal{X}_{p,\rho}}\E\| x_k^\delta-\hat{x}\|^2 &= \sup_{\hat{x}\in\mathcal{X}_{p,\rho}}\sum_{j=1}^k \frac{\E(y^\delta-\hat{y},u_j)^2}{\sigma_j^2} + \sup_{\hat{x}\in\mathcal{X}_{p,\rho}}\sum_{j=k+1}^\infty(\hat{x},v_j)^2\\
            &= \delta^2 \sum_{j=1}^k e^{aj} + \sup_{\substack{\xi\in\mathcal{X}\\\|\xi\|\le \rho}} \sum_{j=k+1}^\infty (-\log(\sigma_j^2))^{-p}(\xi,v_j)^2\\
            &= \delta^2\frac{e^{a(k+1)}}{e-1} + (a(k+1))^{-p} \rho^2.
\end{align*}

We solve the minimization problem by substituting real-valued $x$ for $k+1$ for a moment and solve by standard means (the second derivative is positive, so if suffices to set the first derivative of the right hand side to zero and solve for $x$). We obtain the equation

\begin{align*}
 a\delta^2 \frac{e^{ax}}{e-1} -p a^{-p} x^{-(p+1)} \rho^2 \stackrel{!}{=} 0 \quad \Longrightarrow x^{p+1}e^{ax} = \frac{p(e-1)}{a^{p+1}} \frac{\rho^2}{\delta^2}.
\end{align*}

By Proposition \ref{prop1:th3} (with $b=p+1$ and $y = \frac{p(e-1)}{a^{p+1}} \frac{\rho^2}{\delta^2})$), the unique solution $x_{opt}$ fulfills

$$x_{opt} = \frac{1}{a}\log(y) - \frac{1}{a} \log\left(\left(\frac{1}{a}\log(y)\right)^{p+1}\right)+o(1)$$

 for $y\to\infty$ (which corresponds to $\delta/\rho\to 0$). We handle the data propagation error (variance) and the approximation error (bias) separately. For the approximation error,

\begin{align*}
\rho^2(ax_{opt})^{-p} &= \rho^2\left(\log(y) - \log\left(\left(\frac{1}{a}\log(y)\right)^{p+1}\right)\right)^{-p}(1+o(1)) = \rho^2\log(y)^{-p}(1+o(1))\\
&= \rho^2\left(\log\left(\frac{p(e-1)}{a^{p+1}}\right) + \log\left(\frac{\rho^2}{\delta^2}\right)\right)^{-p}\left(1+o(1)\right)=\rho^2\left(-\log\left(\frac{\delta^2}{\rho^2}\right)\right)^{-p}(1+o(1))
\end{align*}

for $\frac{\delta}{\rho}\to 0$. For the data propagation error

\begin{align*}
 \frac{\delta^2}{e-1} e^{a x_{opt}} &= \frac{\delta^2}{e-1} \frac{y}{\left(\frac{1}{a} \log(y)\right)^{p+1}}(1+o(1)) = p \rho^2 \log\left(\frac{p(e-1)}{a^{p+1}} \frac{\rho^2}{\delta^2}\right)^{-(p+1)} (1+o(1))\\
 &=  p \rho^2\left(-\log\left(\frac{\delta^2}{\rho^2}\right)\right)^{-(p+1)}(1+o(1))= \rho^2 \left(-\log\left(\frac{\delta^2}{\rho^2}\right)\right)^{-p} o(1)
\end{align*}

for $\frac{\delta}{\rho}\to 0$. The above rates stay the same, if we change $x_{opt}$ to any $x\in[x_{opt}-1,x_{opt}+1]$. Therefore, 

$$\inf_{k\in\N} \sup_{\hat{x}\in\mathcal{X}_{p,\rho}}\E\|x_k^\delta-\hat{x}\|^2 = \rho^2\left(-\log\left(\frac{\delta^2}{\rho^2}\right)\right)^{-p}o(1) + \rho^2\left(-\log\left(\frac{\delta^2}{\rho^2}\right)\right)^{-p}(1+o(1)) = \rho^2\left(-\log\left(\frac{\delta^2}{\rho^2}\right)\right)^{-p}(1+o(1))$$

as $\delta/\rho\to \infty$.

\subsection{Proof of Theorem \ref{th4}}

We argue along the lines of the proof of Theorem \ref{th2} and set

\begin{equation}\label{eq1:th4}
m_{opt}=m_{opt}(\delta,\rho,p,a):=\frac{1}{a}\log\left(\frac{\rho^2}{\delta^2}\right) - \frac{1}{a}\log\left(\left(\frac{1}{a}\log\left(\frac{\rho^2}{\delta^2}\right)\right)^{p+1}\right). 
\end{equation}

Again, the stopping index $k_{dp}^\delta(m)$ is in essence bounded by $m_{opt}$.

\begin{proposition}\label{prop1:th4}
It holds that

$$k_{dp}^\delta(m)\chi_{\Omega_{\delta/\rho}} \le \max\left(1, m_{opt}(\delta,\rho,p,a) + C_{a,p,\tau} + o(1)\right)\quad\mbox{for all }m\in\N,$$

where $C_{a,p,\tau}=\frac{1}{a}\log\left(\frac{4}{(\tau-1)^2a^p}\right)+1$ and $m_{opt}$ given in \eqref{eq1:th4}.

\end{proposition}

\begin{proof}
Since $k_{dp}^\delta(m)\le m$ we can assume that $m\ge m_{opt}$. We write $k=k^\delta_{dp}(m)$ for short and argue as in the proof of Proposition \ref{prop1}. Assuming $k\ge 2$ we  obtain

\begin{align}
\tau \sqrt{m}\delta\chi_{\Omega_{\delta/\rho}}\le &\sqrt{\sum_{j=k}^m(\hat{y},u_j)^2} + \sqrt{\sum_{j=k}^m(y^\delta-\hat{y},u_j)^2}\chi_{\Omega_{\delta/\rho}} \le \sqrt{\sum_{j=k}^m \sigma_j^2(-\log(\sigma_j^2))^{-p}(\xi,v_j)^2} + \frac{\tau+1}{2}\sqrt{m}\delta\\
 \le &(ak)^{-\frac{p}{2}}e^{-\frac{ak}{2}} \rho + \frac{\tau+1}{2}\sqrt{m}\delta\\
\Longrightarrow &k^pe^{ak}\chi_{\Omega_{\delta/\rho}} \le  \frac{4}{a^p(\tau-1)^2} \frac{\rho^2}{m\delta^2} \le a e^{C_{a,p,\tau}-1}\frac{\rho^2}{m_{opt}\delta^2}.
\end{align}

We write $C=e^{aC_{a,p,\tau}-1}$ for short and apply Proposition \ref{prop1:th3}. Because of the monotonicity of $f$ there thus holds

$$k \chi_{\Omega_{\delta/\rho}}\le \frac{1}{a}\log\left(C\frac{\rho^2}{m_{opt}\delta^2}\right) - \frac{1}{a}\log\left(\left(\frac{1}{a}\log\left(C\frac{\rho^2}{m_{opt}\delta^2}\right)\right)^{p}\right)+1.$$

The summand $1$ is due to $k\in\N$. We now use  the fact that, for (positive) functions $f,g$ with $g=o(f)$ there holds

$$\log(f+g) = \log(f) + \log(f+g)-\log(f) = \log(f) + \log\left(1+\frac{g}{f}\right) = \log(f) + o(1),$$

to obtain

\begin{align*}
& \frac{1}{a}\log\left(C\frac{\rho^2}{m_{opt}\delta^2}\right) - \frac{1}{a}\log\left(\left(\frac{1}{a}\log\left(C\frac{\rho^2}{m_{opt}\delta^2}\right)\right)^p\right)+1\\
\le &\frac{1}{a}\log\left(\frac{\rho^2}{\delta^2}\right) +\frac{1}{a}\log(C) - \frac{1}{a}\log(m_{opt}) - \frac{1}{a}\log\left(\left(\frac{1}{a}\log\left(\frac{\rho^2}{\delta^2}\right) + \frac{1}{a}\log\left(\frac{C}{m_{opt}}\right)\right)^p\right)+1\\
= &\frac{1}{a}\log\left(\frac{\rho^2}{\delta^2}\right) - \frac{1}{a}\log\left(\frac{1}{a}\log\left(\frac{\rho^2}{\delta^2}\right) - \frac{1}{a}\log\left(\left(\frac{1}{a}\log\left(\frac{\rho^2}{\delta^2}\right)\right)^{p+1}\right)\right) - \frac{1}{a}\log\left(\left(\frac{1}{a}\log\left(\frac{\rho^2}{\delta^2}\right)\right)^p\right) + C_{a,p,\tau} + o(1)\\
= &\frac{1}{a}\log\left(\frac{\rho^2}{\delta^2}\right) - \frac{1}{a}\log\left(\frac{1}{a}\log\left(\frac{\rho^2}{\delta^2}\right)\right) - \frac{1}{a}\log\left(\left(\frac{1}{a}\log\left(\frac{\rho^2}{\delta^2}\right)\right)^p\right) + C_{a,p,\tau} + o(1)\\
=&\frac{1}{a}\log\left(\frac{\rho^2}{\delta^2}\right) - \frac{1}{a}\log\left(\left(\frac{1}{a}\log\left(\frac{\rho^2}{\delta^2}\right)\right)^{p+1}\right) + C_{a,p,\tau}+o(1)
\end{align*}

for $\delta/\rho\to 0$. This proves the claim.

\end{proof}

As in the proof of Theorem \ref{th2}, we will use this fact to bound the data propagation error. 

\begin{proposition}\label{prop3:th4}
 It holds that
 
 $$\mathbb{P}\left(\sum_{j=1}^{k_{dp}^\delta} \frac{(y^\delta-\hat{y},u_j)^2}{\sigma_j^2} \le \rho^2\left(-\log\left(\frac{\delta^2}{\rho^2}\right)\right)^{-p} o(1)\right)\to 1$$
 
 for $\delta/\rho\to 0$.
\end{proposition}

\begin{proof}
Because of Proposition \ref{prop1:th4} it holds that

\begin{align*}
&\mathbb{P}\left(\sum_{j=1}^{k_{dp}^\delta} \frac{(y^\delta-\hat{y},u_j)^2}{\sigma_j^2} \le \rho^2\left(-\log\left(\frac{\delta^2}{\rho^2}\right)\right)^{-p} o(1)\right)\\
\ge &\mathbb{P}\left(\sum_{j=1}^{m_{opt}+C_{a,p,\tau}+o(1)} \frac{(y^\delta-\hat{y},u_j)^2}{\sigma_j^2} \le \rho^2\left(-\log\left(\frac{\delta^2}{\rho^2}\right)\right)^{-p} o(1),~\Omega_{\delta/\rho}\right)\\
\ge &1 - \mathbb{P}\left(\sum_{j=1}^{m_{opt}+C_{a,p\tau}+o(1)} \frac{(y^\delta-\hat{y},u_j)^2}{\sigma_j^2} > \rho^2\left(-\log\left(\frac{\delta^2}{\rho^2}\right)\right)^{-p} o(1)\right) - \mathbb{P}\left(\Omega_{\delta/\rho}^C\right).
\end{align*}

By Proposition \ref{prop1} there holds $\mathbb{P}\left(\Omega_{\delta/\rho}\right)\to 1$ as $\delta/\rho\to 0$. Now Markov's inequality  and the definition of $m_{opt}$ \eqref{eq1:th4} imply

\begin{align*}
&\mathbb{P}\left(\sum_{j=1}^{m_{opt}+C_{a,p\tau}+o(1)} \frac{(y^\delta-\hat{y},u_j)^2}{\sigma_j^2} > \rho^2\left(-\log\left(\frac{\delta^2}{\rho^2}\right)\right)^{-p} o(1)\right) \le \frac{ \sum_{j=1}^{m_{opt}+C_{a,p,\tau}+o(1)} e^{aj}\E(Z,u_j)^2\delta^2}{\rho^2\left(-\log\left(\frac{\delta^2}{\rho^2}\right)\right)^{-p}o(1)}\\
&= \frac{e^{am_{opt}} e^{a(1+C_{a,p,\tau}+o(1))} \delta^2}{\rho^2\left(-\log\left(\frac{\delta^2}{\rho^2}\right)\right)^{-p}o(1)} = \frac{\frac{\rho^2}{\delta^2} \left(\frac{1}{a}\log\left(\frac{\rho^2}{\delta^2}\right)\right)^{-(p+1)}\delta^2}{\rho^2\left(\log\left(\frac{\rho^2}{\delta^2}\right)\right)^{-p}o(1)}e^{a(1+C_{a,p,\tau}+o(1))}\\
 = &\log\left(\frac{\rho^2}{\delta^2}\right)^{-1} \frac{a^{p+1}e^{a(1+C_{a,p,\tau}+o(1))}}{o(1)} \to 0
\end{align*}

as $\delta/\rho\to 0$ (since the $o(1)$ in the nominator can be chosen such that it converges to $0$ arbitrarily slowly). This proves the claim of the Proposition.

\end{proof}

Regarding the approximation error we again rely on $k_{dp}^\delta\ge k_{dp}^\delta(m)$ and show that for $m=m_{opt}$ the approximation error is asymptotically optimal.

\begin{proposition}\label{prop2:th4}
There holds

$$\sqrt{\sum_{j=k_{dp}^\delta(m_{opt})+1}^\infty(\hat{x},v_j)^2}\chi_{\Omega_{\delta/\rho}} \le \rho\left(-\log\left(\frac{\delta^2}{\rho^2}\right)\right)^{-\frac{p}{2}}(1+o(1))$$

as $\delta/\rho\to0$, with $m_{opt}$ given in \eqref{eq1:th4}.

\end{proposition}

\begin{proof}
Let $\varepsilon>0$ be arbitrary. We write $k=k_{dp}^\delta(m_{opt})+1$ and $m=m_{opt}$. First of all, there exist $\alpha=\alpha(\hat{x})$ and $\beta=\beta(\hat{x})$ with $\alpha,\beta\ge 0$ and $\alpha+\beta\le 1$ and

$$\sum_{j=k}^m(\hat{x},v_j)^2 \le \alpha \rho^2\qquad \mbox{and}\qquad \sum_{j=m+1}^\infty(\hat{x},v_j)^2 \le \beta \rho^2.$$

We apply Proposition 2 of \cite{hohage2000regularization} (and use equation (10) therein) to $x^\dagger=\sum_{j=k}^m(\hat{x},v_j)^2$ and $T=K$ and obtain

$$\sum_{j=k}^m(\hat{x},v_j)^2 \le (a+\varepsilon) \rho^2 \left(-\log\left(\frac{\sum_{j=k}^m(\hat{y},u_j)^2}{(\alpha+\varepsilon)\rho^2}\right)\right)^{-p}(1+o(1))$$

for $\delta/\rho\to0$. By definition of $k=k_{dp}^\delta(m_{opt})$ there holds

$$\sum_{j=k}^m(\hat{y},u_j)^2\chi_{\Omega_{\delta/\rho}}\le 2 \sum_{j=k}^m(y^\delta,u_j)^2 + 2 \sum_{j=k}^m(y^\delta-\hat{y},u_j)^2\chi_{\Omega_{\delta/\rho}} \le 2\tau^2 m \delta^2 + 2 \left(\frac{\tau+1}{2}\right)^2 m \delta^2 \le 4 \tau^2 m \delta^2.$$

Therefore,

\begin{align*}
&\sum_{j=k}^m(\hat{x},v_j)^2 \chi_{\Omega_{\delta/\rho}} \le (\alpha+\varepsilon)\rho^2\left(-\log\left(\frac{4\tau^2m \delta^2}{(\alpha+\varepsilon)\rho^2}\right)\right)^{-p}(1+o(1))\\
\le&(\alpha+\varepsilon)\rho^2\left(-\log\left(\frac{\delta^2}{(\alpha+\varepsilon)\rho^2}\right) - \log\left(\frac{4\tau^2}{a}\left(\log\left(\frac{\rho^2}{\delta^2}\right) - \log\left(\left(\frac{1}{a}\log\left(\frac{\rho^2}{\delta^2}\right)\right)^{p+1}\right)\right)\right)\right)^{-p}(1+o(1))\\
\le &(\alpha+\varepsilon)\rho^2\left(-\log\left(\frac{\delta^2}{(a+\varepsilon)\rho^2}\right)\right)^{-p}(1+o(1)). 
\end{align*}

 Moreover,

$$\sum_{j=m+1}^\infty (\hat{x},v_j)^2 \le \left(-\log(\sigma^2_{m+1})\right)^{-p} \sum_{j=m+1}^\infty (\xi,v_j)^2 \le \left(-\log\left( a m\right)\right)^{-p}\beta \rho^2= \beta \rho^2 \left(-\log\left(\frac{\delta^2}{\rho^2}\right)\right)^{-p}(1+o(1)).$$

Adding the estimates we obtain

\begin{align*}
&\sum_{j=k}^\infty(\hat{x},v_j)^2 \chi_{\Omega_{\delta/\rho}} = \sum_{j=k}^m(\hat{x},v_j)^2+\sum_{j=m+1}^\infty(\hat{x},v_j)^2\\
\le &(\alpha+\varepsilon)\rho^2\left(-\log\left(\frac{\delta^2}{(\alpha+\varepsilon)\rho^2}\right)\right)^{-p}(1+o(1)) + \beta\rho^2\left(-\log\left(\frac{\delta^2}{\rho^2}\right)\right)^{-p}(1+o(1))\\
\le &(1+\varepsilon)\rho^2\left(-\log\left(\frac{\delta^2}{\rho^2}\right)\right)^{-p}(1+o(1))\\
&\qquad + (\alpha+\varepsilon)\rho^2\left(\left(-\log\left(\frac{\delta^2}{(\alpha+\varepsilon)\rho^2}\right)\right)^{-p} - \left(-\log\left(\frac{\delta^2}{\rho^2}\right)\right)^{-p}\right)(1+o(1)).
\end{align*}

We bound the last term. For $\frac{\delta^2}{\rho^2}<\varepsilon$ there holds 

\begin{align*}
&\sup_{\alpha\in[0,1]}(\alpha+\varepsilon)\rho^2\left(\left(-\log\left(\frac{\delta^2}{(\alpha+\varepsilon)\rho^2}\right)\right)^{-p} - \left(-\log\left(\frac{\delta^2}{\rho^2}\right)\right)^{-p}\right)\\
\le &(1+\varepsilon)\rho^2\sup_{\alpha\in[0,1]}\left(\frac{1}{\left(\log(a+\varepsilon) +\log\left(\frac{\delta^2}{\rho^2}\right)\right)^p} - \frac{1}{\left(\log\left(\frac{\rho^2}{\delta^2}\right)\right)^p}\right)\\
\le &(1+\varepsilon)\rho^2\left(\frac{1}{\left(\log(\varepsilon) +\log\left(\frac{\delta^2}{\rho^2}\right)\right)^p} - \frac{1}{\left(\log\left(\frac{\rho^2}{\delta^2}\right)\right)^p}\right)=(1+\varepsilon)\rho^2x^{-p} \frac{x^p-(c+x)^p}{(c+x)^p} 
\end{align*}

with $x=\log\left(\frac{\rho^2}{\delta^2}\right)$ and $c=\log(\varepsilon)$, where we used $\alpha\le 1$ and $\alpha> 0$ in the first and second step respectively. Finally,

$$\lim_{x\to\infty} \frac{x^p-(x+c)^p}{(x+c)^p} = \lim_{x\to\infty} \left(\frac{1}{1+\frac{c}{x}}\right)^p - 1  = 0$$

and we obtain

$$\sum_{j=k}^\infty(\hat{x},v_j)^2\chi_{\Omega_{\delta/\rho}} \le (1+\varepsilon)\rho^2\left(-\log\left(\frac{\delta^2}{\rho^2}\right)\right)^{-p}(1+o(1))$$

for $\delta/\rho\to 0$. This finishes the proof, since $\varepsilon$ was arbitrary.

\end{proof}

Finally, combining Proposition \ref{prop2:th4} and \ref{prop3:th4} concludes the proof of Theorem \ref{th4}.

\bibliographystyle{unsrt}
\bibliography{references}
\end{document}